\documentclass[11pt]{article}
\usepackage[latin1]{inputenc}
\usepackage{amssymb}
\usepackage{amsfonts}
\usepackage{amsmath}
\usepackage{amssymb}
\usepackage{amsthm}
\usepackage{color}
\usepackage{longtable}
\input{epsf.sty}
\usepackage{graphicx}

\newtheorem{remark}{Remark}[section]
\newtheorem{example}{Example}[section]

\newtheorem{proposition}{Proposition}[section]

\numberwithin{equation}{section} \numberwithin{theorem}{section}

\newcommand{\red}{\textcolor[rgb]{1.00,0.00,0.00}}

\title{\huge \bf On sums of dependent random lifetimes under the Time Transformed Exponential model}
\author{{\bf Jorge
		Navarro$^{(1)}$\footnote{Corresponding author; \
			tel.: +34 868883509 ;  \ fax: +34 868884182} \  Franco Pellerey$^{(2)}$ and Julio Mulero$^{(3)}$}\\
	$^{(1)}$ Departamento de Estadística e Investigación Operativa\\
	Universidad de Murcia, 30100 Murcia, Spain. \\ ORCID ID: 0000-0003-2822-915X\\Email: jorgenav@um.es\\ $^{(2)}$ Dipartimento di Scienze Matematiche\\
	Politecnico di Torino,	10129 Torino, Italy\\
	ORCID ID: 0000-0002-8983-855X\\ Email: franco.pellerey@polito.it \\ $^{(3)}$ Departamento de Matemáticas\\ Universidad de
	Alicante 
	 03080 Alicante, SPAIN \\
	ORCID ID:  0000-0001-5949-7611\\
	Email: julio.mulero@ua.es \
}
\date{}

\begin{document}
%
%
\maketitle
\newpage

\begin{abstract}
Considered a pair of random lifetimes whose dependence is described by a Time Transformed Exponential model, we provide analytical expressions for the distribution of their sum. These expressions are obtained by using a representation of the joint distribution in terms of multivariate distortions, which is an alternative approach to the classical copula representation. Since this approach allows to obtain conditional distributions and their inverses in simple form, then it is also shown how it can be used to predict the value of the sum from the value of one of the variables (or vice versa) by using quantile regression techniques. \\
	

\noindent
\end{abstract}
%

\section{Introduction}

Let $\mathbf{X}=(X_1,X_2)$ be a pair of dependent lifetimes. The vector $\mathbf{X}$ is said to be described by a \textit{Time Transformed
Exponential model} (shortly, TTE model) if its joint survival function $\bar{F}$ can be written as
\begin{equation}\label{TTE}
\mathbf{\bar{F}}(x_1,x_2)=\bar G(R_1(x_1)+R_2(x_2)), \ x_1,x_2 \ge 0,
\end{equation}
for a suitable one-dimensional, continuous, convex and strictly decreasing survival function $\bar{G}$ and two suitable continuous and
strictly increasing functions $R_i:[0,+\infty)\rightarrow [0,+\infty)$ such that $R_i(0)=0$ and $\lim_{x\rightarrow
\infty}R_i(x)=\infty$, for $i=1,2$. Clearly, the marginal survival functions for the lifetimes $X_i$ are given by $\bar{F}_i(x_i)=\bar{G}(R_i(x_i))$, $x_i \ge 0$, $i=1,2$.

TTE models have been considered in literature as an appropriate manner to describe bivariate lifetimes (see, e.g., \cite{BS05,MPR2010,NM20} and references therein). Their main
characteristic is that they ``separate", in a sense, aging of single lifetimes through the functions $R_i$, and dependence properties through $\bar{G}$, the copula $\widehat C$ being a transformation of $\bar G$ only (see Eq. \eqref{2-1} below and the reference above for details).
This model is of interest in a variety of applicative fields since it is equivalent to the random frailty model, which assumes that the two lifetimes are conditionally independent given a random parameter that represents the risk due to a common environment; the well-known proportional hazard rate Cox model, where the proportional factor is not fixed but random, is an example. In  this case, the different choices for the function $\bar{G}$ are obtained just by changing the distribution of random parameter.

For a number of applicative purposes, one can be interested in the sum $S=X_1+X_2$. This happens, for example, in considering the total lifetime in stand-by systems, where a component is replaced by a new one under the same environmental stress after its failure, or in insurance theory, where the sum of two depended claims, due to common risks, may be evaluated. In this case, because of the dependence between $X_1$ and $X_2$, the classical convolution can not be applied to determine the distribution of $S$, and C-convolutions must be used (see, e.g., \cite{BLS15,CGM16,NS20} for definition and examples of C-convolutions). That is, one can calculate the survival function of $S$ as
\begin{equation}\label{C-conv}
\bar{F}_S(s)=\Pr(S>s)=
\int_{-\infty}^{\infty} f_1(x)\ \partial_{1}\widehat C(\bar F_1(x),\bar F_2(s-x))dx,
\end{equation}
where $\bar F_1$ and $\bar F_2$ are the marginal survival functions of $X_1$ and $X_2$, respectively, $f_1$ is the density function of $X_1$ (assuming its existence) and $\hat C$ is the survival copula of the vector $\mathbf{\mathbf{X}}$. Here, $\partial_{1}\hat C$ means the partial derivative of $\widehat C$ with respect to its first argument. Note that, in particular, for nonnegative random variables Eq. \eqref{C-conv} reduces to
\begin{align*}
\bar{F}_S(s)&
=\bar F_1(s)+
\int_{0}^{s} f_1(x)\  \partial_{1}\widehat C(\bar F_1(x),\bar F_2(s-x))dx, \ \ \ t \ge 0.
\end{align*}

Since usually the integrals appearing in previous formulas are not easy to be solved, we describe in this paper an alternative tool to deal with the sum $S=X_1+X_2$ that can be used when the joint distribution of $\mathbf{X}$ is defined as in Eq. \eqref{TTE}. This approach is based on an alternative representation for the survival function of $\mathbf{X}$, which make use of the distortion representations of multivariate distributions recently introduced in \cite{NCLD20}, whose definition is provided in the next section.
The advantage of this approach is twofold: it is particularly useful when the inverse of $\bar G$ is not available in closed form, thus also $\widehat C$, and it also provides simple representations of the conditional distribution of $S$ given one of the $X_i$, and of its inverse, so that one can use it to predict the value of the sum from the value of one of the variables (or vice versa) by using quantile regression techniques. The purpose of this paper is to describe such an approach.

The rest of the paper is structured as follows. Basic definitions, notations and some preliminary results are introduced in Section 2. The main results for the representation of the distribution of the sum $S$ are provided in Section 3, while examples of their application in prediction are presented in Section 4 and Section 5.

Throughout the paper the notions increasing and decreasing are used in a wide sense, that is, they mean non-decreasing and non-increasing, respectively, and we say that $f$ is increasing (decreasing) if $f(\mathbf{x})\leq f(\mathbf{y})$ for all $\mathbf{x}\leq \mathbf{y}$ (where this last inequality means that for every $i$-th component of the vectors one has $x_i\leq y_i$).
Also, if $f$ is a real valued function in more than one variables, then $\partial_i f$ denotes the partial derivative of $f$ with respect to its $i$-th variable. Analogously,  $\partial_{i,j} f=\partial_i \partial_j f$ and so on.  Whenever we use a partial derivative we are tacitly assuming that it exists.

\section{Notation and preliminary results}

To simplify the notation we just consider here the bivariate case; the extension to the $n$-dimensional case is straightforward. Moreover, we just consider nonnegative random variables with absolutely continuous distributions even if many of the properties described below can be extended to arbitrary random variables.

Thus, let $\mathbf{X}=(X_1,X_2)$ be a  random vector with two possibly dependent nonnegative random variables having an absolutely continuous joint distribution function $\mathbf{F}$ and marginal distributions $F_1$ and $F_2$. Let $\mathbf{f}$ be the joint probability density function (PDF) of $(X_1,X_2)$ and let $f_1$ and $f_2$ be the PDFs of $X_1$ and $X_2$, respectively. Then it is well known (see, e.g., \cite{N06}) that, from Sklar's Theorem, there exists a unique absolutely continuous copula $C$ such that $\mathbf{F}$  can be written as
\begin{equation}\label{Cop}
\mathbf{F}(x_1,x_2)=\Pr(X_1\leq x_1,X_2\leq x_2)=C(F_1(x_1), F_2(x_2))
\end{equation}
for all $x_1,x_2$. As a consequence, the PDF function can be obtained as
\begin{equation*}\label{c}
\mathbf{f}(x_1,x_2)=f_1(x_1)f_2(x_2)c(F_1(x_1), F_2(x_2)),
\end{equation*}
where $c:=\partial_{1,2}C$ is the PDF of the copula $C$. A similar representation holds for the joint survival function
$$	\mathbf{\bar F}(x_1,x_2)=\Pr(X_1>x_1,X_2>x_2)=\widehat C(\bar F_1(x_1),\bar F_2(x_2))$$
for all $x_1,x_2$, where $\bar F_1(x_1)=\Pr(X_1>x_1)$ and $\bar F_2(x_2)=\Pr(X_2>x_2)$ are the marginal survival functions and $\widehat C$ is another suitable copula, called {\it survival copula}.

In the particular case of TTE models, i.e., in the case that the joint survival function $\mathbf{\bar F}$ is defined as in Eq. \eqref{TTE}, then the corresponding survival copula $\widehat C$ is the \textit{strict bivariate Archimedean}  copula (see, e.g.,  \cite{N06}, p. 112) defined as
\begin{equation*}\label{2-1}
\widehat C(u_1,u_2)=\bar G(\bar G^{-1}(u_1)+\bar G^{-1}(u_2)) \
\end{equation*}
for all  $u_1,u_2 \in [0,1]^2$. This model contains many families of copulas (see,  \cite{N06}, p. 117), thus it is a very general dependence model. The inverse function $\bar G^{-1}$ is called the {\it additive generator} of the copula..

However, an alternative representation of $\mathbf{\bar F}$ based on distortion representations of multivariate distributions can be given. In some cases, and for specific applications, such an alternative representation can be preferable to classical copula approach.

For it first recall that a function $d:[0,1]\to[0,1]$ is said to be a {\it distortion function} if it is continuous, increasing  and satisfies $d(0)=0$ and $d(1)=1$. If $G$ is a distribution function, we say that $F$ is a {\it distorted distribution} from $G$ if there exists a distortion function $d$ such that $F(x)=d(G(x))$ for all $x$, and similarly for the survival functions. This kind of representations were introduced in the theory of decision under risk (see e.g. \cite{W96,Y87}) and they were also applied in the fields of coherent systems, order statistics and conditional distributions (see, e.g., \cite{NASS13,NS18} and the references therein).

These representations were further extended to the multivariate case in the recent paper \cite{NCLD20}. According to what defined there, and restricting to the bivariate case, a function $D:\mathbb{R}^2\to \mathbb{R}$ is a {\it bivariate distortion} if it is a continuous $2$-dimensional distribution with support included in $[0,1]^2$, and a bivariate distribution function $\mathbf{F}$ is a distortion of the univariate distribution functions $H_1$ and $H_2$ if there exists a bivariate distortion $D$ such that
\begin{equation}\label{D}
\mathbf{F}(x_1,x_2)=D(H_1(x_1),H_2(x_2))
\end{equation}
for all $x_1,x_2$.

This representation is similar to the copula representation, but here the $H_i$ are not necessarily the marginal distribution of $\mathbf{X}$ and $D$ is not necessarily a copula. Actually, in some situations, we can choose a common univariate distribution $H=H_1=H_2$, and some examples will be provided later (see also \cite{N21,NCLD20}). Moreover, if $D$ has uniform univariate marginal distributions over the interval $(0,1)$, then $D$ is a copula, $H_1,H_2$ are the marginal distributions and \eqref{D} is the same as the copula representation \eqref{Cop} (but only in this case).

The main properties of model \eqref{D} were given in \cite{NCLD20} and they are very similar to that of copulas. For example, if $D$ is a distortion function, then the right-hand side of \eqref{D} defines a proper multivariate distribution function for any univariate distribution functions $H_1$ and $H_2$. Moreover, a similar representation holds for the joint survival function, that is, one can write
\begin{equation}\label{genGK}
\mathbf{\bar F}(x_1,x_2)=\widehat  D(\bar H_1(x_1),\bar H_2(x_2)),
\end{equation}
where $\mathbf{\bar F}(x_1,x_2)=Pr(X_1>x_1,X_2>x_2)$,  $\bar H_i=1-H_i$ for $i=1,2$, and $\widehat D$ is another suitable distortion function.

For TTE model note that, defining
\begin{equation}\label{Hi}
\bar H_i(x_i)=\exp(-R_i(x_i)),
\end{equation}
one has
$$
\mathbf{\bar F}(x_1,x_2)=\bar G(R_1(x_1)+R_2(x_2))=\bar G(-\ln\bar H_1(x_1)-\ln \bar H_2(x_2))=\widehat D(\bar H_1(x_1),\bar H_2(x_2))
$$
for all $x_1,x_2 \ge 0$, where
\begin{equation}\label{hatD-2}
\widehat D(u,v)=\bar G(-\ln(uv)), \ \ u,v \in[0,1].
\end{equation}
The function $\widehat D$ satisfies the property to be a bivariate distortion if $\bar G$ satisfies the properties mentioned above, i.e., if it is an absolutely continuous strictly decreasing convex function in $[0,\infty)$ with $\bar G(0)=1$ and $\bar G(\infty)=0$. Note that if we add $\bar G(t)=1$ for $t<0$, then $\bar G$ is the survival function of a nonnegative random variable. Also note that $\bar H_1$ and $\bar H_2$ are two arbitrary survival functions satisfying $\bar H_1(0)=\bar H_2(0)=1$. Thus, a representation through multivariate distortions as in \eqref{D} holds for the TTE models, with $\widehat D$ defined as in \eqref{hatD-2}.

It must be pointed out that with this representation the marginal survival functions $\bar F_i, i=1,2,$ are not explicitly displayed, but can be obtained as
$$\bar F_i(x_i)=\bar G(-\ln\bar H_i(x_i))=\widehat D(\bar H_i(x_i),1)=\widetilde d(\bar H_i(x_i)), \ \ x_i \ge 0,$$
where $\widetilde d(u)=\bar G(-\ln u),$ $u \in [0,1]$, is a univariate distortion function. Finally, note that the representation through the multivariate distortion \eqref{hatD-2} and the univariate survivals \eqref{Hi} is a copula representation if and only if $\widehat D(u,1)=u,$
that is, $\bar G(-\ln(u))=u$ for $0\leq u\leq 1$. This property leads to $\bar G(x)=\exp(-x)$ for $x\geq 0$ and  $\widehat D(u,v)=uv$ for $u,v\in[0,1]$ which is the product copula that represents the independence case. For other (non-exponential) survival functions $\bar G$, we obtain models with dependent variables, whose dependence is described by $\bar G$.

As an interesting particular case, this dependence model includes the one recently proposed in \cite{GK} for nonnegative random variables, which is characterized (see Proposition 3.1 in \cite{GK}) by the joint survival function
	\begin{equation}\label{GK-model}
	\mathbf{\bar F}(x_1,x_2)=\bar G (\alpha x_1+\beta x_2)
	\end{equation}
	for $x_1,x_2\geq 0$, where $\alpha,\beta>0$ are two positive scale parameters and $\bar G$ satisfies the above mentioned properties. This model, that from now on will be referred as \emph{GK-model} (where the letters G and K indicates the initials of the authors Genev and Kolev of reference \cite{GK}) is an extension of the well known Schur-constant model which is obtained when $\alpha=\beta$ (see \cite{CS94} and references therein). Properties of this model and of the corresponding sum $X_1+X_2$ are studied also in \cite{PN21}.
It must be observed that the marginal survival functions are $\bar F_1(x_1)=\bar G(\alpha x_1)$ and $\bar F_2(x_2)=\bar G(\beta x_2)$ for $x_1,x_2\geq 0$, and both of them belong to the scale parameter model defined by $\bar G$. Actually, this model is obtained by the distortion of univariate exponential distributions, i.e., if \eqref{GK-model} holds, then
$$\mathbf{\bar F}(x_1,x_2)=\widehat D(\bar H_1(x_1),\bar H_2(x_2) )$$
for all $x_1,x_2\geq 0$, where $\bar H_1(x_1)=\exp(-\alpha x_1)$, $\bar H_2(x_2)=\exp(-\beta x_2)$ and $\widehat D$ is defined as in Eq. \eqref{hatD-2}.

\section{Distribution and conditional distribution of the sum}	

In this section we use the distortion representation \eqref{genGK}, with the multivariate distortion $\widehat D$ defined as in \eqref{hatD-2}, to study the sum $S=X_1+X_2$ under the dependence model defined in the preceding section. As a consequence, we also obtain the analogous properties for the GK-model, i.e., the generalization \eqref{GK-model} of the Schur-constant model.

\begin{proposition}
If \eqref{genGK} and \eqref{hatD-2} hold for $(X_1,X_2)$ and $S=X_1+X_2$, then the joint PDF of $(X_1, S)$ is
\begin{equation}\label{g}
	\mathbf{g}(x,s)=r_1(x)r_2(s-x)
\bar G''\left(-\ln\bar H_1(x)-\ln\bar H_2(s-x)\right)
\end{equation}
for $0\leq x\leq s$ (zero elsewhere), where $r_i=(-\ln\bar H_i)'$ is the hazard rate function of $\bar H_i$ for $i=1,2$.
\end{proposition}
\begin{proof}
From \eqref{genGK}, the joint PDF of $(X_1,X_2)$ is
$$\mathbf{f}(x_1,x_2)=h_1(x_1)h_2(x_2)\widehat d(\bar H_1(x_1),\bar H_2(x_2))$$
for $x_1,x_2\geq 0$, where $h_i=-\bar H_i'$ and $\widehat d=\partial_{1,2}\widehat D$. Then the joint PDF of $(X_1,S)$ is
$$\mathbf{g}(x,s)=\mathbf{f}(x,s-x)=h_1(x)h_2(s-x)\widehat d(\bar H_1(x),\bar H_2(s-x))
$$
for $0\leq x\leq s$. The PDF of   our specific distortion function $\widehat D$ is
$$\widehat d(u,v)=\frac 1 {uv} \bar G''(-\ln(uv))$$
and
$$\mathbf{g}(x,s)=\frac{h_1(x)h_2(s-x)}{\bar H_1(x)\bar H_2(s-x)}
\bar G''\left(-\ln\bar H_1(x)-\ln\bar H_2(s-x)\right)
$$
for $0\leq x\leq s$ which concludes the proof.
\end{proof}

\begin{remark}
In particular, for the GK-model in \eqref{GK-model}, that is, with exponential survival functions $H_1$ and $H_2$ with shape parameters (hazard rates) $\alpha$ and $\beta$, the PDF reduces to
$$\mathbf{g}(x,s)=\alpha\beta \bar G''( (\alpha-\beta)x+\beta s )$$
for $0\leq x\leq s$ (zero elsewhere).
Therefore its joint distribution function is
	\begin{align*}
	\mathbf{G}(x,s)
	&=-\int_0^x \int_y^s \alpha\beta G''( (\alpha-\beta)y+\beta t )dtdy\\
	&= \int_0^x  \alpha G'( \alpha y )dy-\int_0^x  \alpha G'( (\alpha-\beta)y+\beta s )dy
	\end{align*}
where $G=1-\bar G$. To solve this integral we consider two cases. If $\alpha \neq \beta$, then
	\begin{align}\label{G-GK1}
	\mathbf{G}(x,s)
	&=G(\alpha x)-\frac {\alpha}{\alpha-\beta}  G ((\alpha-\beta)x+\beta s )+ \frac {\alpha}{\alpha-\beta}  G(\beta s )
	\end{align}
	while  if $ \alpha =\beta$, then
	\begin{align}\label{G-GK2}
	\mathbf{G}(x,s)
	&= G(\alpha x)-\alpha x  G'( \alpha s )
	\end{align}
	for $0\leq x\leq s$. 	In both cases, \eqref{G-GK1} and \eqref{G-GK2} can be represented as distorted distributions from  $G$ by replacing $x$ with $G^{-1}(G(x))$ and  $s$ with $G^{-1}(G(s))$.
	
In particular, as immediate consequence one can obtain the distribution of the sum (C-convolution) for the GK-model as	$$F_S(s)=\lim_{x\to \infty}\mathbf{G}(x,s)=\mathbf{G}(s,s).$$
	If $\alpha \neq \beta$, then
	$$
	F_S(s)= \frac {\alpha}{\alpha-\beta}  G(\beta s )-\frac {\beta}{\alpha-\beta}  G (\alpha s )
	$$
	or if $\alpha=\beta$, then
	$$F_S(s)= G(\alpha s)-\alpha s  G'( \alpha s )$$
	for $s\geq 0$. Note that the first expression is a negative mixture  (i.e. a linear combination with a negative weight) with PDF
\begin{equation}\label{fs1}
f_S(s)=\frac {\alpha\beta}{\beta-\alpha}  \left[ g (\alpha s )-g(\beta s )\right]
\end{equation}
	for $s\geq 0$, where $g=G'$. In the second case, one gets
\begin{equation}\label{fs2}
f_S(s)= -\alpha^2 s G'' (\alpha s)
\end{equation}
	for $s\geq 0$, which is the expression in Remark 2.7 of \cite{CS94} (i.e. for the Schur-constant model).
\end{remark}

The joint survival function of $(X_1, S)$ under the model \eqref{genGK} and \eqref{hatD-2} in the non-exponential case for the $H_i$ is obtained in the following proposition. Unfortunately, an explicit expression can not be provided in general, but it is available in some cases, or easily available numerically (see the examples in the next sections).

\begin{proposition}
	If \eqref{genGK} and \eqref{hatD-2} hold for $(X_1,X_2)$ and $S=X_1+X_2$, then  the joint survival function of $(X_1, S)$ is
	$$\mathbf{\bar G}(x,s)=\bar G(-\ln \bar H_1(s) )+\int_x^s r_1(y)
	g\left(-\ln\bar H_1(y)-\ln\bar H_2(s-y)\right)dy
	$$
	for $0\leq x\leq s$, where $g=-\bar G'$ is the PDF of $\bar G$ and $r_i=(-\ln\bar H_i)'$ is the hazard rate function of $\bar H_i$ for $i=1,2$.
\end{proposition}
\begin{proof}
	From Eq. \eqref{g} for the PDF of $(X_1, S)$ we get
\begin{align*}
\mathbf{\bar G}(x,s)
&=\int_x^s \int_s^\infty  \mathbf{g}(y,t)
dtdy+ \int_s^\infty \int_y^\infty  \mathbf{g}(y,t)dtdy\\
&=\int_x^s \int_s^\infty
r_1(y)r_2(t-y)
\bar G''(-\ln\bar H_1(y)-\ln\bar H_2(t-y))dtdy\\
&\quad +\int_s^\infty \int_y^\infty
r_1(y)r_2(t-y)
\bar G''(-\ln\bar H_1(y)-\ln\bar H_2(t-y))dt dy\\
&=-\int_x^s r_1(y)
\bar G'(-\ln\bar H_1(y)-\ln\bar H_2(s-y))dy
-\int_s^\infty r_1(y)
\bar G'(-\ln\bar H_1(y))dy\\
&=\bar G(-\ln \bar H_1(s) )+\int_x^s r_1(y)
g(-\ln\bar H_1(y)-\ln\bar H_2(s-y))dy
\end{align*}
which concludes the proof.
\end{proof}

Therefore, the survival function of $S$ can be obtained as
\begin{equation}\label{G2}
\bar F_S(s)=\mathbf{\bar G}(0,s)=\bar G(-\ln \bar H_1(s) )+\int_0^s r_1(y)g\left(-\ln\bar H_1(y)-\ln\bar H_2(s-y)\right)dy
\end{equation}
and its PDF as $f_S(s)=- \partial_2 \mathbf{\bar G}(0,s), \ s \ge 0$.

To get the explicit expression for $\bar F_S$ we need to explicitate $\bar G$ and/or $\bar H_i$ and to solve this integral, eventually numerically. For example, if $\bar H_i(x)=\exp(-x)$ for $x\geq 0$, then
$$\bar F_S(s)=\mathbf{\bar G}(0,s)=\bar G(s)+\int_0^s g(y+s-y)dy=\bar G(s)+ sg(s)
$$
and $f_S(s)=sg'(s)$ for $s\geq 0$ which is the expression in Remark 2.7 of \cite{CS94} for the Schur-constant model.

\section{Predictions}

The purpose of this section is to show how to predict the  value of the sum $S=X_1+X_2$ from $X_1=x$ or vice versa by making use of the results in the previous section. To this purpose we need the conditional distribution of $(S|X_1=x)$ in the TTE dependence model, that is obtained in the following proposition.

\begin{proposition}
	If \eqref{genGK} and \eqref{hatD-2} hold  for $(X_1,X_2)$ and $S=X_1+X_2$, then  the PDF of  $( S|X_1=x)$ is
$$f_{S|X_1}(s|x)=-r_2(s-x)\frac{
	g'(-\ln\bar H_1(x)-\ln\bar H_2(s-x))}{g(-\ln\bar H_1(x))}$$
	and its distribution function is
\begin{equation}\label{G2|1}	
F_{S|X_1}(s|x)	=1-\frac{
		g(-\ln\bar H_1(x)-\ln\bar H_2(s-x))}{g(-\ln\bar H_1(x))}
\end{equation}
	for $0\leq x\leq s$, where $g=-\bar G'$ and $r_2=(-\ln\bar H_2)'$ is the hazard rate function of $\bar H_2$.
\end{proposition}
\begin{proof}

From \eqref{g}, the PDF of $(X_1,S)$ is
$$\mathbf{g}(x,s)=-r_1(x)r_2(s-x)
g'(-\ln\bar H_1(x)-\ln\bar H_2(s-x))$$
for $0\leq x\leq s$. Moreover, the first marginal survival function is
$$\bar F_1(x)=\Pr(X_1>x)=\mathbf{G}(x,0)=\bar G(-\ln \bar H_1(x))
$$
and its PDF is
$f_1(x)=r_1(x)g(-\ln \bar H_1(x))$,
for $x\geq 0$.

Hence, the  PDF of $(S|X_1=x)$ for $x\geq 0$ such that $g_1(x)>0$, can be obtained as
$$f_{S|X_1}(s|x)=\frac{\mathbf{g}(x,s)}{f_1(x)}=-r_2(s-x)\frac{
g'(-\ln\bar H_1(x)-\ln\bar H_2(s-x))}{g(-\ln\bar H_1(x))}$$
for $s\geq x$ (zero elsewhere).

Then the associated distribution function is
\begin{align*}
F_{S|X_1}(s|x)
	&=\int_{x}^s f_{S|X_1}(t|x)dt\\ &=-\int_{x}^s  r_2(t-x)\frac{
	g'(-\ln\bar H_1(x)-\ln\bar H_2(t-x))}{g(-\ln\bar H_1(x))}dt\\
&=\left[-\frac{
	g(-\ln\bar H_1(x)-\ln\bar H_2(t-x))}{g(-\ln\bar H_1(x))}\right]_{t=x}^s\\
&=1-\frac{
	g(-\ln\bar H_1(x)-\ln\bar H_2(s-x))}{g(-\ln\bar H_1(x))}
\end{align*}
for $s\geq x\geq 0$ and we conclude the proof.
\end{proof}

Hence, the conditional survival function is
$$\bar F_{S|X_1}(s|x)	=\frac{
	g(-\ln\bar H_1(x)-\ln\bar H_2(s-x))}{g(-\ln\bar H_1(x))}$$
Clearly, this is a distortion representation from $\bar H_2(s-x)$, since
\begin{equation*}\label{G21C}
\bar F_{S|X_1}(s|x)=d_{S|X_1}(\bar H_2(s-x)|\bar H_1(x))
\end{equation*}
for $s\geq x>0$, where
$$d_{S|X_1}(v|u)=\frac{
	g(-\ln uv) }{g(-\ln u)} $$
for $v\in[0,1]$  is a distortion function for all $0<u<1$.

Note that the inverse function of $\bar{F}_{S|X_1}(v|u)$ (i.e. its quantile function)
can be obtained from the inverse functions  of $g$ and $\bar H_2$ as
\begin{equation}\label{IG2|1}
{\bar F}^{-1}_{S|X_1}(v|u)(q|x)	=x+\bar H_2^{-1} \left( \frac{\exp \left(-g^{-1}\left( qg(-\ln\bar H_1(x))\right)\right)}{\bar H_1(x)}\right)
\end{equation}
for $0<q<1$. The inverse function of ${F}_{S|X_1}$ can be obtained in a similar way.

One can thus predict $S$ from $X_1=x$ by using  the quantile (or median) regression curve with
$$m_{S|X_1}(x):= \bar{F}_{S|X_1}^{-1}(0.5|x).$$

Moreover, one can compute the centered $p$ confidence bands for these estimations as
$$I_p(x):=\left[\bar{F}_{S|X_1}^{-1}\left( \frac{1+p}2\mid x\right),\bar {F}_{S|X_1}^{-1}\left(\frac  {1-p}2\mid x\right) \right].$$
For example, the $p=90\%$ centered confidence band for $S$
is
$$I_{0.9}(x):=\left[\bar{F}_{S|X_1}^{-1}(0.95|x), \bar{F}_{S|X_1}^{-1}( 0.05|x) \right].$$
Such an interval is computed below in some illustrative examples.

\begin{remark}
In particular, for the GK-model in \eqref{GK-model} we get

\begin{equation}\label{G21new}
{\bar F}_{S|X_1}(s|x)	=\frac{
	g(\alpha x+\beta(s-x))}{g(\alpha x)}=\frac{
	g((\alpha-\beta)x+\beta s)}{g(\alpha x)}
\end{equation}
and
\begin{equation}\label{IG21}
\bar F^{-1}_{S|X_1}(q|x)=\frac{\beta-\alpha}{\beta} x+ \frac{1}{\beta}  g^{-1}\left( qg(\alpha x) \right)
\end{equation}
for $s\geq x\geq 0$ and $0<q<1$. Note that these expressions hold both for $\alpha\neq \beta$ and for $\alpha=\beta$.
\end{remark}

The other conditional distribution can be obtained in a similar manner. However, it is more difficult to get an explicit expression since we need the PDF {$f_S(s)$ of $S$. It is stated in the following proposition.

\begin{proposition}
	If \eqref{genGK} and \eqref{hatD-2} hold  for $(X_1,X_2)$ and $S=X_1+X_2$, then  the PDF of  $( X_1|S=s)$ is
	$${f}_{X_1|S}(x|s)=-\frac{r_1(x)r_2(s-x)
			g'\left(-\ln\bar H_1(x)-\ln\bar H_2(s-x)\right)}{f_S(s)}$$
	and its distribution function is
	\begin{equation}\label{G1|2}	
	{F}_{X_1|S}(x|s)=\red{-}\int_{0}^x \frac{r_1(t)r_2(s-t)
		g'\left(-\ln\bar H_1(t)-\ln\bar H_2(s-t)\right)}{f_S(s)}dt
	\end{equation}
	for $0\leq x\leq s$, where $g=-\bar G'$ and $r_i=(-\ln\bar H_i)'$ is the hazard rate function of $\bar H_i$ for $i=1,2$, and
	\begin{equation}\label{g2}
		f_S(s)=-\int_0^s r_1(x)r_2(s-x)
		g'(-\ln\bar H_1(x)-\ln\bar H_2(s-x))dx.
	\end{equation}
\end{proposition}
\begin{proof}
	
From \eqref{g}, the PDF of $(X_1,S)$ is
$$\mathbf{g}(x,s)=-r_1(x)r_2(s-x)
	g'(-\ln\bar H_1(x)-\ln\bar H_2(s-x))$$
for $0\leq x\leq s$. Its second  marginal survival function was obtained in \eqref{G2}. It can also be obtained as in \eqref{g2}.
	
	Hence, the conditional PDF of $(X_1|S=s)$ is
$${f}_{S|X_1}(x|s)=\frac{\mathbf{g}(x,s)}{f_S(s)}=-\frac{r_1(x)r_2(s-x)
	g'(-\ln\bar H_1(x)-\ln\bar H_2(s-x))}{f_S(s)}.$$

Then the associated distribution function is the one given in \eqref{G1|2} for $0\leq x\leq s$ and the assertion is proved.
\end{proof}

In particular, for the GK-model we have the following explicit expressions.

\begin{proposition}
	If \eqref{GK-model} holds for $(X_1,X_2)$ and $S=X_1+X_2$, then  the distribution function of  $( X_1|S=s)$ is
	\begin{equation}\label{G1|2a}	
	{F}_{X_1|S}(x|s)=\frac{g((\alpha-\beta)x+\beta s)-g(\beta s)}{ g(\alpha s)-g(\beta s)}
	\end{equation}
when $\alpha\neq \beta$ and
\begin{equation}\label{G1|2b}	
{F}_{X_1|S}(x|s)=\frac{x}{s}
\end{equation}
when $\alpha=\beta$, for $0\leq x\leq s$, where $g=-\bar G'$ and $\alpha,\beta>0$ are the scale parameters in \eqref{GK-model}.
\end{proposition}
\begin{proof}

From the preceding proposition we have
$$\mathbf{g}(x,s)=-r_1(x)r_2(s-x)
g'(-\ln\bar H_1(x)-\ln\bar H_2(s-x))= -\alpha\beta g'((\alpha-\beta)x+\beta s)$$
for $0\leq x\leq s$ (zero elsewhere). Its second  marginal PDF function $f_S$ was obtained in \eqref{fs1} ($\alpha\neq\beta$) and in \eqref{fs2} ($\alpha=\beta$).
	
In the first case we get	
$${f}_{X_1|S}(x|s)=\frac{\mathbf{g}(x,s)}{f_S(s)}
=(\alpha-\beta) \frac{g'((\alpha-\beta)x+\beta s)}{ g(\alpha s)-g(\beta s)}$$
and in the second
$${f}_{X_1|S}(x|s)=\frac{\mathbf{g}(x,s)}{f_S(s)}
=\frac{-\alpha^2g'(\alpha s)}{ -\alpha^2sg'(\alpha s)}=\frac{1}{s}$$
for $0\leq x\leq s$.

Then the associated distribution functions are
\begin{align*}
{F}_{X_1|S}(x|s)
&=\int_{0}^x (\alpha-\beta) \frac{g'((\alpha-\beta)t+\beta s)}{ g(\alpha s)-g(\beta s)}dt\\
&=\left[  \frac{g'((\alpha-\beta)t+\beta s)}{ g(\alpha s)-g(\beta s)}\right]_{t=0}^x\\
&=  \frac{g((\alpha-\beta)x+\beta s)-g(\beta s)}{ g(\alpha s)-g(\beta s)}
\end{align*}
(in the first case) and
\begin{align*}
{F}_{X_1|S}(x|s)&=\int_{0}^x \frac 1 s dt=\frac x s
\end{align*}
(in the second case) for $0\leq x\leq s$. 
\end{proof}

Note that the expression \eqref{G1|2b} was obtained previously in Proposition 2.3 of \cite{CS94} for the Schur-constant model (which is equivalent to \eqref{GK-model} with $\alpha=\beta$).

As in the preceding case, expressions \eqref{G1|2a} and \eqref{G1|2b} can be used to obtain quantile regression curves to predict $X_1$ from $S$. An illustrative example is given in the following section. 	In both cases, they can be represented as distorted distributions from  $G$ by replacing $x$ with $G^{-1}(G(x))$ and  $s$ with $G^{-1}(G(s))$.

In the second case ($\alpha=\beta$),  the inverse function is
${F}^{-1}_{X_1|S}(q|s)=qs$ for $0<q<1$ and the trivial median regression curve is $m_{X_1|S}(s)=s/2$ (which in this case coincides with the classic mean regression curve $E(X|S=s)$).

In the first case ($\alpha\neq\beta$), we get
\begin{equation}\label{IG12}
{F}^{-1}_{X_1|S}(q|s)=\frac \beta {\beta -\alpha}s+\frac 1 {\alpha -\beta} g^{-1}\left( q g(\alpha s)+(1-q)g(\beta s)\right)
\end{equation}
for $0<q<1$ and $s>0$. Then the median regression curve is
$$m_{X_1|S}(s)=\frac \beta {\beta -\alpha}s+\frac 1 {\alpha -\beta} g^{-1}\left( \frac 1 2 g(\alpha s)+\frac 1 2g(\beta s)\right).
$$
The confidence bands can be obtained in a similar manner from \eqref{IG12} (see Example \ref{Ex2}).

\section{Examples}

In this section we provide  some examples to illustrate the theoretical findings described in previous sections. In the first one we consider the sum of two dependent random variables satisfying the GK-model proposed in \cite{GK}, i.e., the model \eqref{GK-model}.

\begin{example}\label{Ex1}
Let us assume that $(X_1,X_2)$ satisfies  \eqref{GK-model} for $\alpha\neq \beta$ and   $\bar G(x)=(1+x)^{-\gamma}$ for $x\geq 0$ (Pareto type II survival function) and $\gamma>0$. This model is equivalent to consider an Archimedean Clayton copula with $\theta=1/\gamma$ and Pareto type II marginals. Then, from \eqref{G-GK1}, the joint distribution function of $(X_1,S)$ is
\begin{align*}
\mathbf{G}(x,s)
&= G(\alpha x)-\frac {\alpha}{\alpha-\beta}  G ((\alpha-\beta)x+\beta s )+ \frac {\alpha}{\alpha-\beta}  G(\beta s )\\
&=1-\left( 1+\alpha x\right)^{-\gamma}+\frac {\alpha}{\alpha-\beta}\left( 1+(\alpha-\beta)x+\beta s\right)^{-\gamma} -\frac {\alpha}{\alpha-\beta} \left( 1+\beta s\right)^{-\gamma}
\end{align*}
for $0\leq x\leq s$. Hence, the distribution function $F_S$ of $S$ (i.e. the C-convolution) is
$$
F_S(s)=\mathbf{G}(s,s)=1+\frac {\beta}{\alpha-\beta}  (1+\alpha s)^{-\gamma} -\frac {\alpha}{\alpha-\beta}  (1+\beta s)^{-\gamma}
$$
for $s\geq 0$. Its PDF is
$$f_S(s)=\frac {\alpha \beta \gamma}{\alpha-\beta}  (1+\beta s)^{-\gamma-1}-\frac {\alpha \beta \gamma}{\alpha-\beta}  (1+\alpha s)^{-\gamma-1}  $$
for $s\geq 0$. The distribution of $S$ is a negative mixture of two Pareto type II distributions and so its hazard rate goes to zero when $s\to\infty$ (which is the limit of the hazard rates of the members of the C-convolution). They are plotted in Figure \ref{fig1} (right)  jointly with the associated PDF functions (left) for $\gamma=\alpha=2$ and $\beta=1$. Note that the hazard rates of $X_1$ and $X_2$ are decreasing while the one of $S$ is not monotone (showing that the IFR class is not preserved by the sum of dependent r.v.).

If we want to predict $X_1$ from $S=s$, we need the conditional distribution obtained from  \eqref{G1|2a} as
$$F_{X_1|S}(x|s)=\frac{g((\alpha-\beta)x+\beta s)-g(\beta s)}{ g(\alpha s)-g(\beta s)}=\frac{(1+(\alpha-\beta)x+\beta s)^{-\gamma-1}-(1+\beta s)^{-\gamma-1}}{ (1+\alpha s)^{-\gamma-1}-(1+\beta s)^{-\gamma-1}}
$$
for $0\leq x\leq s$.
Its inverse function is then
$$F^{-1}_{X_1|S}(q|s)=  \frac{-1-\beta s+ \left(q(1+\alpha s)^{-\gamma-1}+(1-q) (1+\beta s)^{-\gamma-1}\right)^{-1/(\gamma+1)}}{\alpha-\beta}
$$
for $0<q<1$. The median regression curve is obtained by replacing $q$ with $1/2$. It is plotted in Figure \ref{fig2},  jointly with a sample from $(X_1,S)$ and the associated $50\%$ and $90\%$ centered confidence bands. We also include there the parametric (left) and non-parametric (right) estimations for these curves. Here, non-parametric means that we use  the linear quantile regression procedure in R.

To estimate the parameters in the model from the sample  we use recall the Kendall's tau coefficient of $(X_1,X_2)$ is
$$\tau=\frac{\theta}{2+\theta}=\frac 1{1+2\gamma}$$
(see \cite{N06}, p. 163). Therefore, $\gamma$ is estimated by
$$\widehat \gamma= \frac{1-\widehat \tau}{2\widehat \tau}=\frac {1-0.158}{2\cdot 0.158}=2.664557.$$
Then we recall that $E(X_1)=1/(\alpha (\gamma-1))$ and $E(X_2)=1/(\beta (\gamma-1))$ to estimate $\alpha$ and $\beta$, obtaining
$$\widehat \alpha=\frac1{(\widehat \gamma-1 )\bar X_1 }=\frac 1{1.664557\cdot 0.3880776}=1.548042$$
and
$$\widehat \beta=\frac1{(\widehat \gamma-1 )\bar X_2 }=\frac 1{1.664557\cdot 0.8674393}=0.6925677.$$

For the non-parametric linear estimators of the quantile regression curves, we used the R library \texttt{quantreg} (see \cite{K05,KB78,N20}). The estimated median regression line to estimate $X_1$ from $S$ obtained from our sample is
$$\widehat m_{X_1|S}(s)=0.09752378 +0.17721635 s.$$

The procedure to predict $S$ from $X_1$ is analogous.

\end{example}


\begin{figure}[t]
	\begin{center}
\includegraphics[scale=0.55]{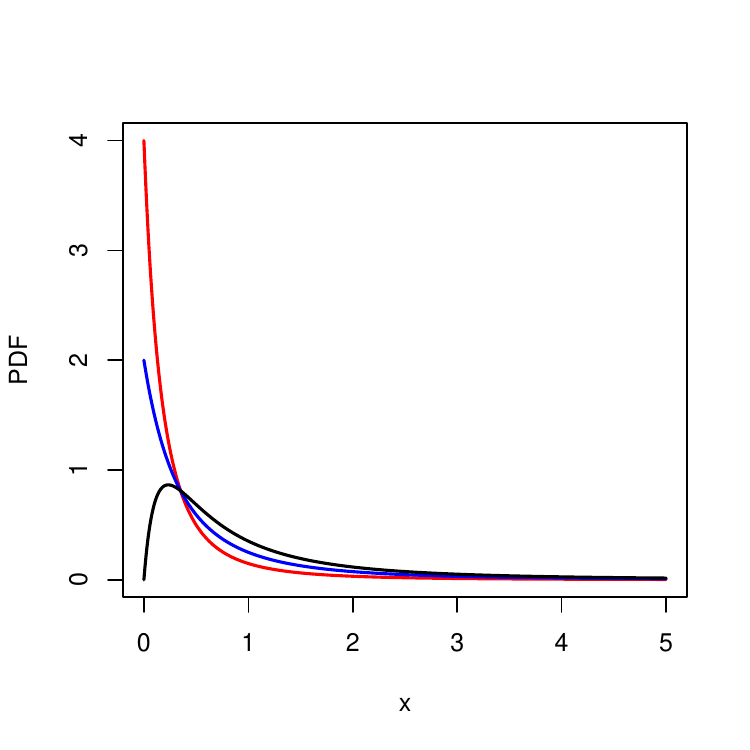}		\includegraphics[scale=0.55]{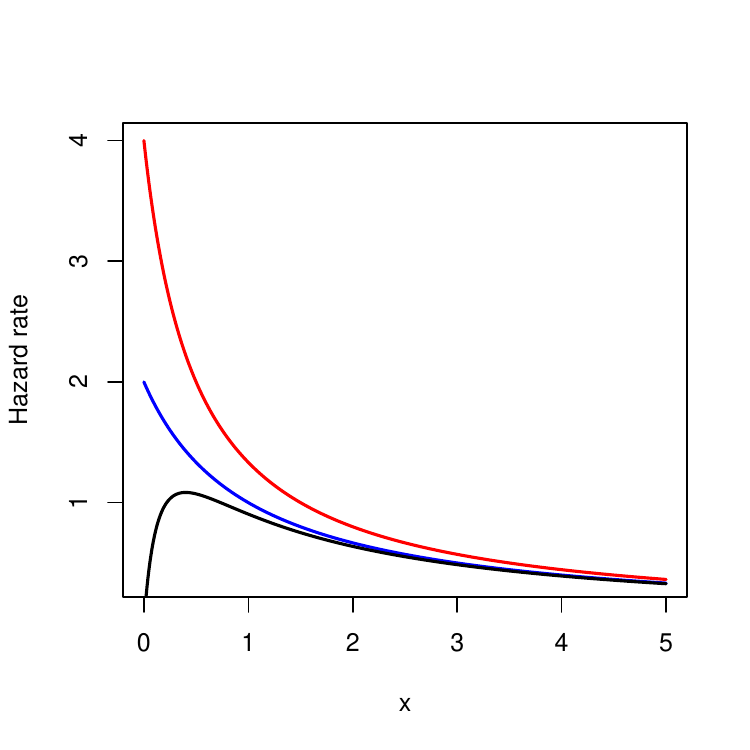}
	\end{center}
	\caption{Probability density (left) and hazard rate (right) functions for $X_1$ (red), $X_2$ (blue) and  $S=X_1+X_2$ (black) under the dependence model \eqref{GK-model} with Pareto type II  marginals studied in Example \ref{Ex1}.}%
	\label{fig1}%
\end{figure}

	\begin{figure}[h!]
		\begin{center}
		\includegraphics[scale=0.55]{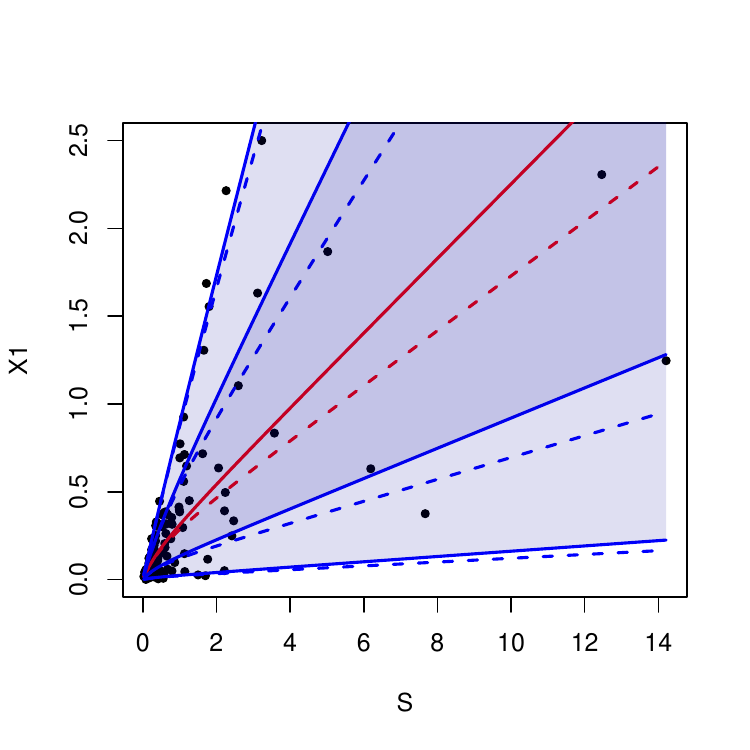}
		\includegraphics[scale=0.55]{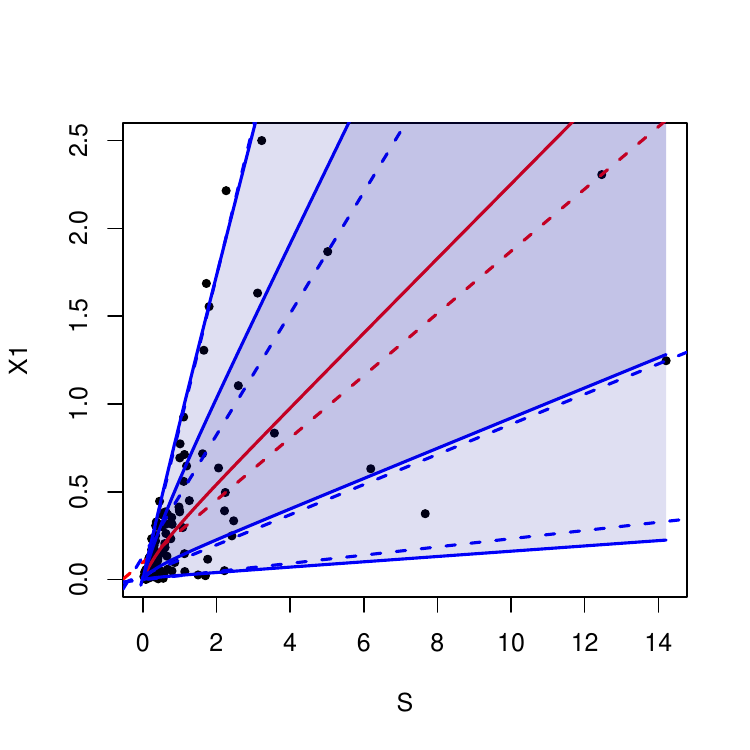}
		\end{center}
		\caption{Scatterplot of the simulated sample  from $(S,X_1)$  in Example \ref{Ex1} jointly with the exact median regression curve (continuous red lines) and the exact $50\%$  and $90\%$  confidence bands (continuous blue lines). The dashed lines represent the estimated curves when the model is known and the parameters are estimated (left) and when the model is unknowns and we use a non-parametric linear quantile regression  estimator (right) from these data.}%
		\label{fig2}%
	\end{figure}

\quad

\quad
	
In the next example we consider the general dependence model defined in \eqref{genGK} and \eqref{hatD-2}. In this case we show how to predict $S$ from $X_1$.

\begin{example}\label{Ex2}
Recall that in \eqref{genGK} we assume that the joint survival  function of $(X_1,X_2)$ can be written as
\begin{equation*}\label{GDmodel}
\mathbf{\bar F}(x_1,x_2)=\widehat D(\bar H_1(x_1), \bar H_2(x_2)),
\end{equation*}
where $\bar H_1$ and  $\bar H_2$ are two absolutely continuous survival functions with $\bar H_1(0)=\bar H_2(0)=1$, while \eqref{hatD-2} asserts that
\begin{equation}\label{DD2}
\widehat D(u,v)=\bar G( - \ln(uv))
\end{equation}
for $u,v\in[0,1]$, where $\bar G$ satisfies the properties stated after Eq. \eqref{hatD-2}. This model is a bivariate distorted distribution, for which the marginal survival functions are
$\bar F_i(x_i)=\bar G(-\ln(\bar H_i(x_i)))$, for $i=1,2.$ Thus, we can use the expressions obtained in Section 4, \eqref{G2|1} and \eqref{IG2|1}, to predict $S$ from $X_1$.


For example, we can choose
$$\bar G(x)=\bar H_1(x)=\bar H_2(x)=c\cdot (1-\Phi(1+x))=c\cdot \Phi(-1-x))$$  for $x\geq 0$, where $\Phi$ is the standard normal distribution and $c=1/\Phi(-1)=6.302974$  (i.e. $G$ is a truncated Normal distribution). Hence $g(x)=c\cdot \phi(1+x)$ where $\phi=\Phi'$ is the PDF of a standard normal distribution. Note that, in this case, the associated Archimedean copula (that we could call  Gaussian Archimedean copula) does not have an explicit expression (since it depends on $\bar G$ and on $\bar G^{-1}$). Thus, this is a practical example where the distortion representation \eqref{DD2} can be used as a proper alternative.

Its inverse functions are
$$\bar G^{-1}(x)=-1-\Phi^{-1}(x/c)$$
and
$$g^{-1}(x)=-1+\left( 2\ln c-\ln(2\pi)-2\ln x \right)^{1/2}.$$
By using these expressions we compute ${\bar F}^{-1}_{S|X_1}$ as in 
\eqref{IG2|1}, obtaining the quantile regression curve plotted in Figure \ref{fig3}, left. The same figure also includes a sample of $n=100$ points from $(X_1,S)$ and the exact centered $50\%$  and $90\%$ (blue) confidence bands. Moreover, it shows the plot of the non-parametric linear quantile estimate (dashed lines) obtained from this sample.

As we know that $X_1<S$, we could also provide bottom  $50\%$ and $90\%$  confidence bands obtained as $\left[x,{\bar F}^{-1}_{S|X_1}(0.5|x)\right]$
and $\left[x,{\bar F}^{-1}_{S|X_1}(0.1|x)\right]$, respectively. They are plotted in Figure \ref{fig3}, right. In this case, the median regression curve is also the upper limit for the $50\%$   confidence band. In our sample we obtain $10$ data above the upper (exact) limit and $46$ above the median regression curve (i.e. $54$ data in the exact bottom $50\%$  confidence band). The estimated median regression line obtained from our sample is
$$\widehat m_{S|X_1}(x)=0.3159734 +0.7284655x$$
for $x\geq 0$.
\end{example}

\begin{figure}[t]
	\begin{center}
\includegraphics[scale=0.55]{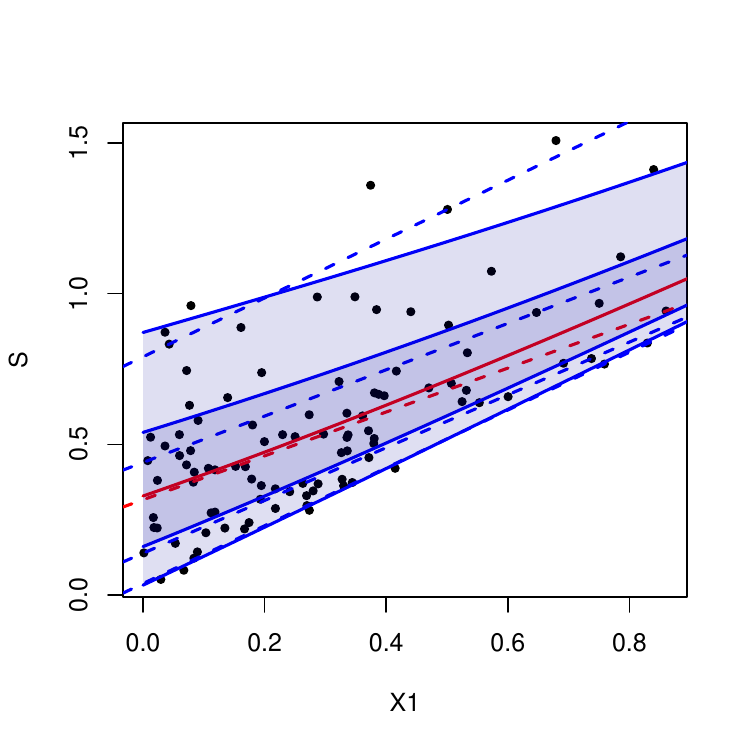}
\includegraphics[scale=0.55]{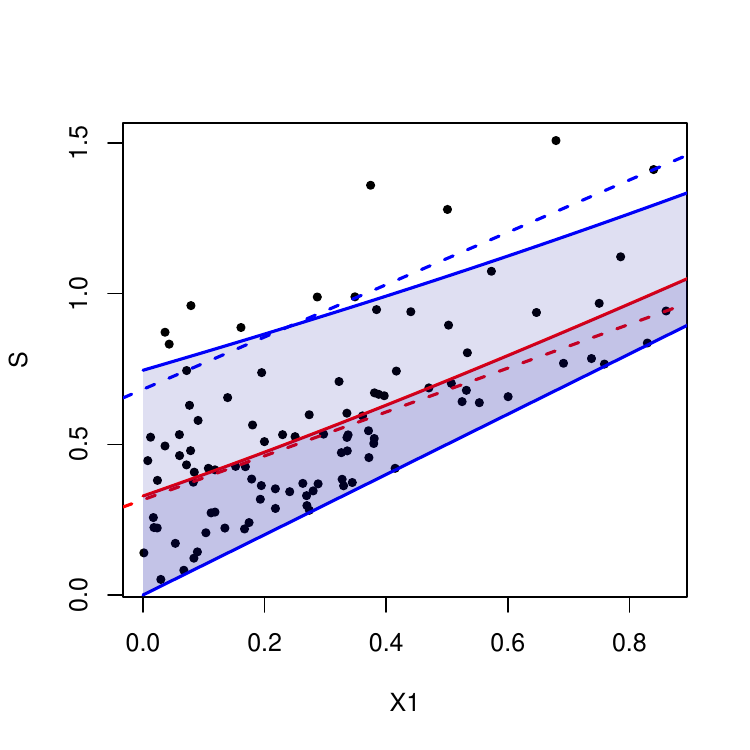}
	\end{center}
	\caption{Scatterplot of the simulated sample  from $(X_1,S)$  in Example \ref{Ex2} jointly with the median regression curve (red) and the centered (left) or bottom (right) $50\%$  and $90\%$  confidence bands (blue). The dashed lines represent the estimated values when we use a linear quantile regression estimator.}%
	\label{fig3}%
\end{figure}

\quad

In the next example we show a case of model \eqref{GK-model} that cannot be represented with an explicit Archimedean copula, thus for which the distortion representations consists in a useful alternative tool. In fact, in this example $\bar G$ is convex and an explicit expression for its inverse is not available. For this model we compute the explicit expressions for the C-convolution and the two conditional survival functions.

\begin{example}\label{ex3}
Consider \eqref{GK-model} with $\alpha\neq \beta$ and the survival function
$$\bar G(x)=\frac{2+x}{2} e^{-x}$$
for $x\geq 0$. Its PDF is
$$g(x)=\frac{1+x}{2} e^{-x}$$
for $x\geq 0$, that is, it is a translated  Gamma (Erlang) distribution. The joint survival function of $(X_1,X_2)$ is
$$\mathbf{\bar F}(x_1,x_2)=\bar G(\alpha x_1+\beta x_2)= \frac{1+\alpha x_1+\beta x_2}2 \exp(-\alpha x_1-\beta x_2)$$
for $x_1,x_2\geq 0$.  The marginals have also translated  Gamma distributions.

The joint distribution of $(X_1,S)$ can be obtained from \eqref{G-GK1}. From this expression, one can obtain the survival function of $S$ (C-convolution) as
\begin{align*}
\bar F_S(s)
&=\frac {\alpha}{\alpha-\beta} \bar G(\beta s )- \frac {\beta}{\alpha-\beta}  \bar G (\alpha s )\\
&=\frac {\alpha}{\alpha-\beta}e^{-\beta s}- \frac {\beta}{\alpha-\beta}  e^{-\alpha s}  + \frac {\alpha \beta s}{2(\alpha-\beta)}\left(e^{-\beta s}-e^{-\alpha s}\right)
\end{align*}
for $s\geq 0$. Note that it is a negative mixture of two translated Gamma distributions.

The conditional survival function  of $(S|X_1=x)$ can be obtained from \eqref{G21new} as
$$\bar F_{S|X_1}(s|x)=\frac{g((\alpha-\beta)x+\beta s)}{g(\alpha x)}=\frac{1+ (\alpha-\beta)x+\beta s}{1+\alpha x}e^{-\beta(s-x)}$$
for $s\geq x$. Analogously,  from \eqref{G1|2a}, the
conditional survival function  of $(X_1|S=s)$ is
$$\bar F_{X_1|S}(x|s)=\frac{g(\alpha s)-g((\alpha-\beta)x+\beta s)}{g(\alpha s)-g(\beta s)}
=\frac{1+\alpha s -(1+(\alpha-\beta)x+\beta s)e^{(\alpha-\beta)(s-x)}}
{1+\alpha s-(1+\beta s)e^{(\alpha-\beta)s}}$$
for $0\leq x\leq s$.

In Figure \ref{fig4} we plot the probability density (left) and  hazard rate (right) functions of $X_1$ (red), $X_2$ (blue) and $S$ (black) when $\alpha=2$ and $\beta=1$. Note that both marginals are IFR and the same holds for $S$. Also note  that the limiting behaviour of the hazard rate of $S$ coincides with that of the best component ($X_2$) in the sum. This is according with the results on mixtures obtained in Lemma 3.3 of \cite{NS06} (or Lemma 4.6 in \cite{NS20}) and  that in Theorem 1 of \cite{BLS15} on usual convolutions.

\end{example}

\begin{figure}[t]
	\begin{center}
		\includegraphics[scale=0.55]{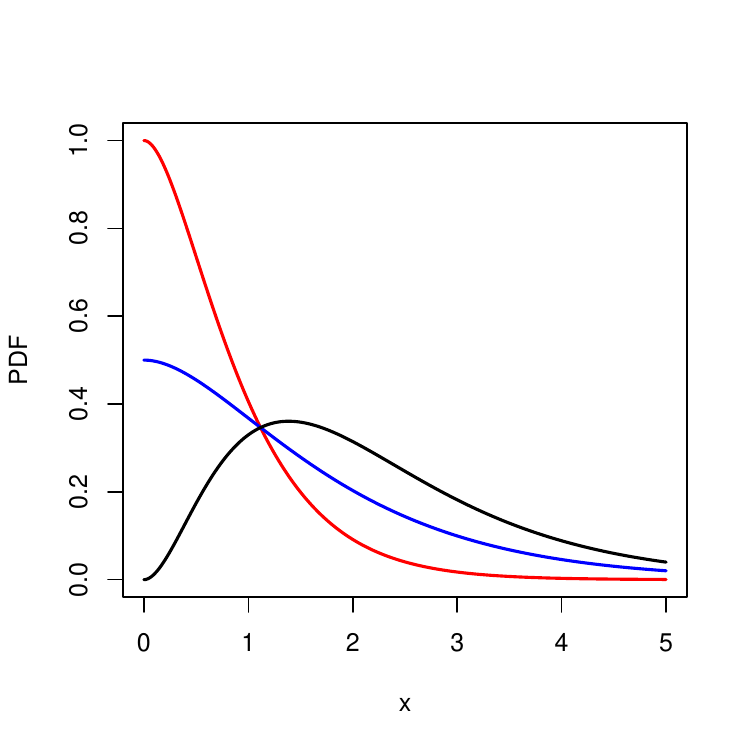}		\includegraphics[scale=0.55]{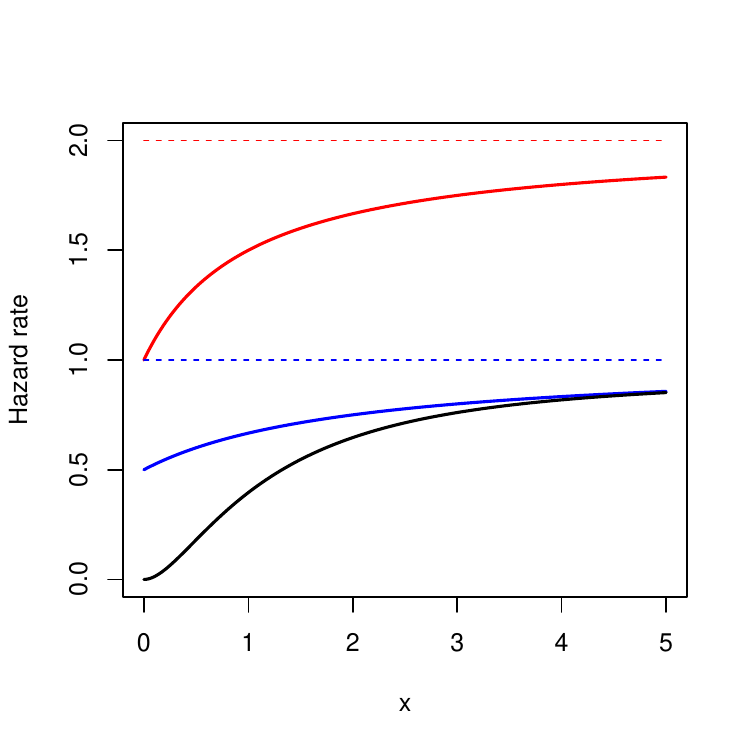}
	\end{center}
	\caption{Probability density (left) and hazard rate (right) functions for $X_1$ (red), $X_2$ (blue) and  $S=X_1+X_2$ (black) under the dependence model \eqref{GK-model} with translated Gamma marginals studied in Example \ref{ex3}. The dashed lines represent the limiting behaviour.}%
	\label{fig4}%
\end{figure}

In the last example we show a case dealing with the GK model \eqref{GK-model} where the inverse of the conditional distribution function $F_{(X_1|S)}$ of $(X_1|S)$ cannot be obtained in a closed form. Then we need to use numerical methods (or implicit functions plots). It also shows that the quantile (median) regression curve $m_{(X_1|S)}(s)=F_{(X_1|S)}(0.5|s)$ is not always increasing.

\begin{example}\label{ex4}
Let us consider the model \eqref{GK-model} which leads to a survival copula in the family of  Gumbel-Barnett copulas (see (4.2.9) in \cite{N06}, p.\ 116). In this case, the additive generator of the copula is $\bar G^{-1}(t)=\ln(1-\theta\ln t)$ for $t\in (0,1]$ and $\theta\in(0,1]$. These copulas are strict Archimedean copulas and the independence (product) copula is obtained for $\theta\to 0$. Hence,
$$\bar G(t)= \exp\left( \frac 1 \theta - \frac 1 \theta e^t  \right)$$
and
$$g(t)=\frac 1 \theta\exp\left( t+\frac 1 \theta - \frac 1 \theta e^t  \right)$$
for $t\geq 0$. Note that the inverse of $g$ has not an explicit form, thus one cannot use \eqref{IG12} to compute the quantile functions of $(X_1|S)$. The same happens in \eqref{IG21} for  the quantile functions of $(S|X_1)$.

However, it is possible to plot the level curves of the conditional distribution function by using \eqref{G1|2a}, obtaining
\begin{equation}\label{G1|2aGB}	
{F}_{(X_1|S)}(x|s)=\frac{g((\alpha-\beta)x+\beta s)-g(\beta s)}{ g(\alpha s)-g(\beta s)}
\end{equation}
when $\alpha\neq \beta$. For example, if we choose $\alpha=3$, $\beta=1$  and $\theta=1$ in \eqref{G1|2aGB}, we get
\begin{equation*}\label{G1|2aGB}	
{F}_{(X_1|S)}(x|s)=\frac{g(2x+ s)-g(s)}{ g(3s)-g(s)}
=\frac{\exp\left( 2x+s+ 1 -   e^{2x+s}  \right) -\exp\left( s+ 1 -   e^{s}  \right)  }{\exp\left( 3s+ 1 -   e^{3s}  \right)-\exp\left( s+ 1 -   e^{s}  \right)}
\end{equation*}
for $0\leq x\leq s$. These level curves for $q=0.05,0.25,0.5,0.75,0.95$ are plotted in Figure \ref{fig5}, left. Note that the median regression curve $m_{(X_1|S)}(s)=F^{-1}_{(X_1|S)}(0.5|s)$ (red line, left) is first increasing and then decreasing. To explain this surprising fact we plot ${F}_{(X_1|S)}(x|s)$ in Figure \ref{fig5}, right, for different values of $s$, where one can observe that these distribution functions are not ordered in $s$, that is, $(X_1|S=s)$ is not stochastically increasing in $s$. Here the greater values for $X_1$ are obtained when $S\approx 0.6$ (green line). Also note that $E(X_2)=3E(X_1)$ and that $X_1$ and $X_2$ are negatively correlated. Therefore, the greater values of $S$ are mainly obtained from the greater values of $X_2$ and the smaller values of $X_1$. Thus $m_{(X_1|S)}$ is decreasing at the end.

Also note that
$$Cov(X_1,S)=Var(X_1)+Cov(X_1,X_2)=Var(X_1)+E(X_1X_2)-E(X_1)E(X_2).$$
Therefore, $Cov(X_1,S)\geq 0$ when $Cov(X_1,X_2)\geq 0$ and, in particular, when $X_1$ and $X_2$ are independent. However, the covariance $Cov(X_1,S)$ will be negative  if $Var(X_1)<-Cov(X_1,X_2)$. In our case, the marginal reliability functions of $X_1$ and $X_2$ are $\bar F_1(t)=\bar G(3t)$ and  $\bar F_2(t)=\bar G(t)$, respectively. Their means are $E(X_1)=0.198782$ and $E(X_2)=0.596347$, their variances $Var(X_1)=0.019589$ and $Var(X_2)=0.176301$ and their covariance  $Cov(X_1,X_2)=-0.029889$. Hence
$$Cov(X_1,S)=Var(X_1)+Cov(X_1,X_2)=0.019589-0.029889=-0.010299<0.$$

\end{example}

\begin{figure}[t]
	\begin{center}
		\includegraphics[scale=0.55]{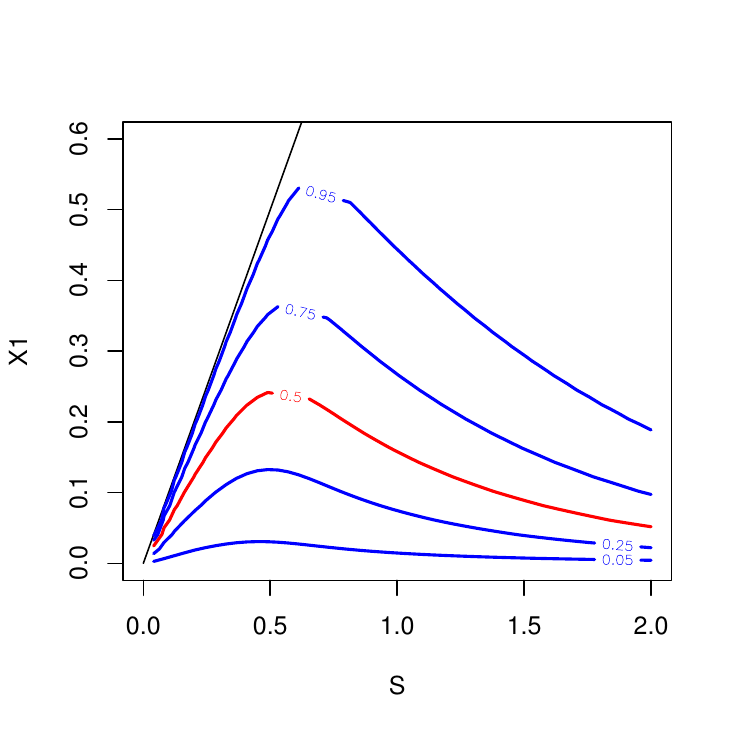}
		\includegraphics[scale=0.55]{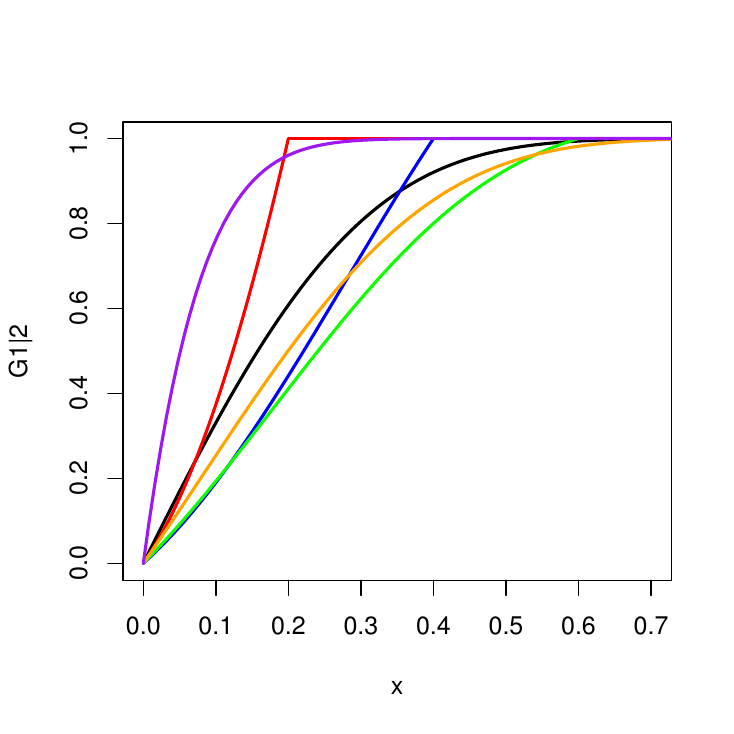}
	\end{center}
	\caption{Median regression curve (red) and  quantile regression curves  (blue) for $q=0.05,0.25,0.75,0.95$
 (left) for $(S,X_1)$  in Example \ref{ex4}. Conditional distribution functions  ${G}_{1|2}(x|s)$ for $s=0.2$ (red), $0.4$ (blue), $0.6$ (green), $0.8$ (orange), $1$ (black) and $2$ (purple). The black line in the left plot represents the line $X_1=S$.  
	}%
	\label{fig5}%
\end{figure}

\section{Conclusions}

We formulated the TTE dependence model by using a distortion representation based on a specific fixed distortion function $\widehat D$. This representation is useful to compute the joint distribution of $X_1$ and the sum $S=X_1+X_2$, as well as to provide expressions for the survival function of $S$ and the conditional distributions of $S$ given $X_1$ or $X_1$ given $S$.  They can be used also to predict one value from the other by using quantile regression. Some examples illustrate these facts, showing that sometimes the classical copula approach can not be applied.

This paper is a first step on applications of distortion representations for dependence models. Thus, there are several tasks for future research. The main one could be to get explicit models by choosing appropriate functions $\bar G$, $\bar H_1$ and $\bar H_2$, to study their main properties and how they fit to real data sets, allowing for the use of the prediction techniques developed here for these data sets. Other interesting questions deal with dependence models for which the multivariate distortion function $\widehat D$ differs from the one in Eq.  \eqref{hatD-2}, or how to get explicit expressions for the multivariate case.

	
\section*{Acknowledgements}


JN and JM are supported  by Ministerio de Ciencia e Innovación  of Spain under grant PID2019-103971GB-I00/AEI/10.13039/501100011033. FP is partially supported by the grant \emph{Progetto di Eccellenza, CUP: E11G18000350001} and by the Italian GNAMPA research group of
INdAM (Istituto Nazionale Di Alta Matematica).

\end{document}